\documentclass[12pt]{article}
\usepackage{amsmath,amsfonts,amssymb,amsthm,amscd}
\title{Remarks on Banach spaces determined by their finite dimensional subspaces}
\date{\today}
\author{Karim Khanaki\thanks{K.Khanaki@gmail.com  \ \ \  Partially supported by IPM grants 93030032 and 93030059}\\Arak University of Technology
 }

\newtheorem{Theorem}{Theorem}[section]

\newtheorem{Definition}[Theorem]{Definition}
\newtheorem{Notation}[Theorem]{Notation}
\newtheorem{Remark}[Theorem]{Remark}
\newtheorem{Lemma}[Theorem]{Lemma}
\newtheorem{Corollary}[Theorem]{Corollary}
\newtheorem{Fact}[Theorem]{Fact}

\newtheorem{Question}[Theorem]{Question}

\newcommand{\sslash}{\mathbin{\mkern-6mu/\mkern-6mu/\mkern-5mu}}

\begin{document}
\maketitle

\begin{abstract} A separable Banach space $X$ is said to be {\em
finitely determined} if for each separable space $Y$ such that
$X$ is finitely representable (f.r.) in $Y$ and $Y$ is f.r. in
$X$ then $Y$ is isometric to $X$. We provide a direct proof
(without model theory) of the fact that
 every finitely determined space $X$ (isometrically) contains
every (separable) space $Y$ which is finitely representable in
$X$. We also point out how a similar argument proves the
Krivine-Maurey theorem on stable Banach spaces, and give the
model theoretic interpretations of some results.
\end{abstract}

\section{Introduction}
This paper is a kind of companion-piece to \cite{K}, although here
we are mainly concerned with direct proofs (without any use of
model theory) of some results of \cite{K} and some new results.
The main results are Theorem~\ref{compact->ell_p} on connections
of isometry groups and existence of classical sequence spaces,
Theorem \ref{fr determined->compact} on isometry group of
`finitely determined spaces',  Corollary~\ref{main1} showing that
the finitely determined spaces contain `good' subspaces, and
Remark~\ref{stable} pointing out that the Krivine-Maurey theorem
has an equivalent formulation in the language of general topology.

We study Banach spaces which have `large' isometry groups. These
spaces are very `symmetrical' and have `good' substructures.
Recall that a  separable Banach space $X$ is said to be {\em
determined by its first order theory} (or {\em
$\aleph_0$-categorical}) provided that $X$ is isometric to every
space $Y$ such that ultrapowers $X_{\cal
U}$ and $Y_{\cal U}$ are isometric for some ultrafilter $\cal U$.  
In fact, a space is determined by its first order theory if it is
the only separable model of its first order theory in the sense
of Continuous Logic (see \cite{BBHU}). In \cite{K}, it is shown
that every $\aleph_0$-categorical space contains some $\ell_p$ or
$c_0$. In this paper we show that every space which is determined
by its finite dimensional subspaces (see Definition~\ref{finitely
determined} below) contains $\ell_p$ for each $p$ in its
spectrum. In fact the latter is a consequence of the prior, and
in this paper we give direct proofs of them and some new results.

One reason for restricting our attention to direct proofs
(without model theory) is to make the paper more accessible to
Banach-theorists and other interested readers.

\subsection{Isomery groups and classical sequences}
 Let $(X,d)$ be a complete metric space and $G$ a
closed subgroup of isometry group $\text{Iso}(X)$ with the
natural topology (i.e. the topology pointwise convergence). For
$x\in X^{\Bbb N}$, we let $[x]$ (or $[x]_G$, if there is a risk of
ambiguity) denote the closure of the orbit $x$ with respect to
the product topology on $X^{\Bbb N}$. We fix a metric inducing
the product topology and it is denoted by $d$ again.

\begin{Definition}[\cite{BT}] \label{oligomorphic}  {\em Let $(X,d)$ be a complete metric space and $G$ a closed
subgroup of isometry group $\text{Iso}(X)$.  We equip the set of
orbit closures
$$X\sslash G=\{[x]_G:x\in X\}$$
with the metric induced from $X$
$$d([x],[y])=\inf\{d(u,v):u\in[x],v\in[y]\}.$$
We say that the action $G$ on $X$ is  {\em approximately
oligomorphic} (or short, $G$ is {\em approximately oligomorphic})
if $X^n\sslash G$ is compact for all $n$. }
\end{Definition}

\begin{Remark} {\em (i) Since elemnts of $G$ are isometry, then
$[x]\cap[y]\neq\emptyset$ if and only if  $[x]=[y]$.
\newline
 (ii) Since $G$ is a subgroup of isometries, $d$ is a
metric on $X\sslash G$., i.e. the triangle inequality holds and
$d([x],[y])=0$ implies $[x]=[y]$.
\newline
 (iii) $G$ is approximately oligomorphic iff
$X^{\Bbb N}\sslash G$ is compact (see the discussion after
Definition~2.1 in \cite{BT}).
\newline (iv) If $(X,d)$ is
complete then $X^{\Bbb N}\sslash G$ is complete. So the latter
space is compact iff it is totally bounded. }
\end{Remark}

Now we can prove some results, after introducing some notations:

\begin{Notation} {\em Let $X$ be a Banach spaces and
$x_1,\ldots,x_n,y_1,\ldots,y_n\in X$. For $\epsilon>0$, we write
$(x_1,\ldots,x_n)\sim_\epsilon(y_1,\ldots,y_n)$ if for each
$r_1,\ldots,r_n\in \Bbb R$,
$$ (1-\frac{1}{\epsilon}) \|\sum_1^n r_i y_{i}\| \le \|\sum_1^n r_i x_{i}\| \le  (1+\frac{1}{\epsilon}) \|\sum_1^n r_i y_{i}\|.$$
We write $(x_1,\ldots,x_n)\sim_\epsilon (e_1,\ldots,e_n)$ (where
$(e_n)$ is the standard basis of $\ell_p$) if for each
$r_1,\ldots,r_n\in \Bbb R$,
$$(1-\frac{1}{\epsilon})\big(\sum_1^n |r_i|^p\big)^{\frac{1}{p}} \le \|\sum_1^n r_i x_{n,i}\| \le (1+\frac{1}{\epsilon})\big(\sum_1^n
|r_i|^p\big)^{\frac{1}{p}}.$$ }
\end{Notation}

We give an easy lemma.

\begin{Lemma} Let $X$ be a Banach space and
$x_1,\ldots,x_n,y_1,\ldots,y_n\in X$ such that
$(x_1,\ldots,x_n)\sim_\epsilon(y_1,\ldots,y_n)$  for some
$\epsilon>0$. Then for every  {\em linear} isometry $g$ on $X$,
$(g(x_1),\ldots,g(x_n))\sim_\epsilon(y_1,\ldots,y_n)$.
\end{Lemma}

\begin{proof} Immadiate, since $g$ is linear and isometry.
\end{proof}

\begin{Theorem} \label{compact->ell_p}   Suppose that $X$ is a separable Banach space and
$G$ is a closed subgroup of {\em linear} isometies on $X$. If
$X^{\Bbb N}\sslash G$ is compact then $X$ contains $c_0$ or
$\ell_p$ (for some $1\le p<\infty$).
\end{Theorem}

\begin{proof}
By Krivine's theorem, there is some $p\in[1,\infty]$ such that
$\ell_p$ is finitely representable in $X$. So, for each $n\in \Bbb
N$, let $x_{n,1},x_{n,2},\ldots,x_{n,n}$ be in $X$ such that
$(x_{n,1},x_{n,2},\ldots,x_{n,n})\sim_{\frac{1}{n}}(e_1,\ldots,e_n)$.
Set $x_n=(x_{n,1},\ldots,x_{n,n}, z_{n+1},z_{n+2},\ldots)$ where
$z_i$ are arbitrary in $X$. By compactness, the sequence
$([x_n])_{n\in\Bbb N}$ has a cluster point $[x]$ in $X^{\Bbb N}
\sslash  G$.  Now, it is easy to verify that
$[x]=[(x_1,x_2,\ldots)]$ contains $\ell_p$. Indeed, for each
$\epsilon>0$, there is a $[x_n]$ such that $d([x_n],[x])\le
\epsilon$. Since elements of $G$ are {\em linear} isometry, as
$(x_{n,1},\ldots,x_{n,n})\sim_\epsilon (x_1,\ldots,x_n)$ we have
$(x_1,\ldots,x_n)\sim_{2\epsilon} (e_1,\ldots,e_n)$. As
$\epsilon$ is arbitrary, the proof is completed.
\end{proof}

For non-specialists in Banach space theory, we mention that a
Banach space $X$ is said to be {\em finitely representable}
(f.r.) in another Banach space $Y$ if for each finite dimensional
subspace $X_n$ of $X$ and each number $\lambda>1$, there is an
isomorphism $T_n$ of $X_n$ into $Y$ for which
$\lambda^{-1}\|x\|\leqslant\|T_n(x)\|\leqslant\lambda\|x\|$ if
$x\in X_n$.

\begin{Remark} {\em  In the above (\ref{compact->ell_p}), by  a simlar
argument, one can show that if a Banach space $Y$  which has a
basis is f.r. in $X$, and $X^{\Bbb N} \sslash G$ is compact, then
$X$ (isometrically) contains $Y$. }
\end{Remark}

\begin{Definition} \label{finitely determined} {\em  A separable Banach space $X$ is said to be
{\em determined by its finite dimensional subspaces} (or short
{\em finitely determined}) if for each separable space $Y$ such
that $X$ is finitely representable (f.r.) in $Y$ and $Y$ is f.r.
in $X$ then $Y$ is isometric to $X$. }
\end{Definition}

Now we recall some facts from general topology. Let $X$ be a
topological space, $(x_i)_{i\in I}$ an indexed family in $X$, and
$\cal U$ an ultrafilter on $I$. For $x\in X$, we write
$\lim_{{\cal U}, i}x_i=x$ and say that $x$ is the $\cal U$-limit
of $(x_i)_{i\in I}$ if for every neighborhood $U$ of $x$, the set
$\{i\in I:x_i\in U\}$ is in $\cal U$. A basic fact in general
topology is that $X$ is compact and Hausdorff iff for every
indexed family $(x_i)_{i\in I}\in X$ and every ultrafilter $\cal
U$ on $I$, the $\cal U$-limit exists and is unique. We use this
fact in the following.

\begin{Theorem} \label{fr determined->compact}  Suppose that $X$ is determined by its finite dimentional subspaces, and $G_L$ the group of all linear isometries on $X$. Then
$X^{\Bbb N} \sslash G_L$  is compact.
\end{Theorem}

\begin{proof} Suppose, if  possible, that $X^{\Bbb N} \sslash G_L$ is not
compact. So, there is an ultrafilter $\cal U$ on $\Bbb N$ and
sequence $\{x_n\}\in X^{\Bbb N} \sslash G$ such that the $\cal
U$-limit does not exist (in $X^{\Bbb N} \sslash G$), this means
that for every $u_n\in x_n$, the $\cal U$-limit of the sequence
$\{u_n\}$ does not exist in $X^{\Bbb N}$. Let $X_{\cal U}$ be the
$\cal U$-ultrapower of $X$ (see \cite{Hei}, for the definition).
By Theorem 6.3 of \cite{Hei}, $X_{\cal U}$ is f.r. in $X$. For
each $n$,  let $(x_{n,1},x_{n,2},\ldots)\in x_n$. Then, for each
$m$, the sequence $\{x_{n,m}\}_{n=1}^\infty$ determines an element
$y_m=(x_{n,m})_{\cal U}$ in $X_{\cal U}$ (see again \cite{Hei}
for definition of elements of $X_{\cal U}$). So,
$\{y_1,y_2,\ldots\}$ is f.r. in $X$. Let $\hat X$ be the Banach
space generated by $X\cup\{y_1,y_2,\ldots\}$. Clearly, $\hat X$
and $X$ are not isometrically isomorphic. Indeed, suppose, if
possible, that $f:\hat X\to X$ is an isometrically isomorphism
and $f(y_n)=z_n$ for all $n$. Since $\lim_{{\cal U},n}x_{n,m}=
y_m$, so $(z_1,z_2,\ldots)$ is the $\cal U$-limit of the sequence
 $\{(f(x_{n,1}),f(x_{n,2}),\ldots)\}_{n=1}^\infty$, but
$(f(x_{n,1}),f(x_{n,2}),\ldots)\in x_n$ for all $n$ (note that
$f$ is a {\em linear isometry}), and  the sequence $\{x_n\}$ was
not convergent. To summarize, the space $\hat X$ is different
from $X$, but $\hat X$ is f.r. in $X$ and $X$ is f.r. in $\hat
X$. This is a contradiction.
\end{proof}

\begin{Corollary} \label{main1} Suppose that $X$ is determined by its finite dimentional subspaces. Then $X$ contains $c_0$ or $\ell_p$.
Moreover, $X$ contains every Banach space which has a basis and
is f.r. in $X$.
\end{Corollary}

\begin{proof} Putting  \ref{compact->ell_p} and \ref{fr
determined->compact} together we get the first part. For the
general case, an argument similar to \ref{compact->ell_p} for any
such space $Y$ works well.
\end{proof}

\begin{Remark} \label{stable} {\em  (i) Recall that two spaces $X$ and $Y$ are called
{\em almost isometric} if for each $\lambda>1$, they are
$\lambda$-isomorphic, that is, there is a linear isomorphism
$T:X\to Y$ such that for all $x\in X$,
$\lambda^{-1}\|x\|\le\|T(x)\|\le \lambda\|x\|$.
 In the above, one can works with {\em almost} linear
isometries instead of linear isometries, and get similar results.
Indeed, we say that $\bar{x}, \bar{y}\in X^{\Bbb N}$ are in the
same class, and write $\bar{x}\sim\bar{y}$, if for each
$\lambda>1$ there is a linear $\lambda$-isomorphism such that
$T(\bar{x})=\bar{y}$. Now, if $X^{\Bbb N}\sslash\sim$ is compact
(with respect to the metric which is defined similiar to
Definition~\ref{oligomorphic}), then $X$ almost isometrically
contains any space $Y$ which has a basis and is f.r. in $X$. To
summarize, we say that a Banach space $X$ is {\em almost
determined by its finite dimensional subspaces} if for any space
$Y$ such that $X$ is f.r. in $Y$ and $Y$ is f.r. in $X$, we have
$X$ and $Y$ are almost isometric. So, every almost determined
space (almost isometrically) contains any space which is f.r. in
the space.
\newline
(ii) By a similar argument, one can show that
$\aleph_0$-categoricity implies compactness of $X^{\Bbb
N}\sslash\text{Aut}(X)$,  where $\text{Aut}(X)$ is the {\em group
of automorphisms  of structure $X$} (see \cite{BBHU} for the
definition). Note that $\text{Aut}(X)$ is a closed subgroup of
$G_L$. So,

\medskip
{\scriptsize finitely determined ~~$\Longrightarrow$~~
$\aleph_0$-categorical ~~$\Longrightarrow$~~  $X^{\Bbb
N}\sslash\text{Aut}(X)$ compact ~~$\Longrightarrow$~~ $X^{\Bbb
N}\sslash G_L$ compact  }
\medskip
\newline
(iii) {\bf (stability):} By a similar argument one can show that a
stable Banach space (in the sense of Krivine and Maurey) almost
isometrically contains some $\ell_p$. Indeed, for a space $X$,
let $X_{\cal U}$ be an ultrapower of $X$, and $G_L$ the ismoetry
group of $X_{\cal U}$. Then we consider the space $X_{\cal
U}\sslash G_L$. Let $A=\{[x]:x\in X^{\Bbb N}\}$. One can prove
that a space $X$ is stable iff the closure of $A$, denoted by
$\bar A$, is compact in $X_{\cal U}\sslash G_L$. Clearly, if
$\bar A$ is compact, then some $\ell_p$ is a limit of a {\em
sequence} in $A$, by Krivine's theorem.

}


\end{Remark}

\section{Model theoretic interpretations}
Recall that a separable Banach space $X$ is said to be {\em
determined by its first order theory} if it is the only separable
model of its first order theory in the sense of Continuous Logic
(see \cite{BBHU}). In this case, $X$ is called {\em
$\aleph_0$-categorical} or {\em separably categorical}.

\begin{Remark}  \label{Ultra & Keisler-Shelah}   {\em  (i) $X$ is f.r. in $Y$ if and only if there is an ultrafilter $\cal U$ such that $X$ is isometric to a subspace of
the ultrapower $Y_{\cal U}$. (See \cite{Hei}, Theorem~6.3.)
\newline
(ii) 
 If $X$ and $Y$ have the same first
order theory, denoted by $Th(X)=Th(Y)$ or $X\equiv Y$, then there
is an ultrafilter $\cal U$ such that the ultrapowers $X_{\cal U}$
and $Y_{\cal U}$ are isometric. (See \cite{BBHU}, Theorem~5.7.)  }
\end{Remark}

\begin{Corollary}  
\label{finite->categorical} Every separable space which is determined by its finite dimensional
subspaces is also determined by its first order theory,
equivalently, it is $\aleph_0$-categorical.
\end{Corollary}

\begin{proof} Putting (i) and (ii) of Remark~\ref{Ultra & Keisler-Shelah} together the proof is completed.
\end{proof}

The converse of the above does not hold in general. For example,
$L_p[0,1]$ is determined by its theory but $\ell_p$ is f.r in
$L_p[0,1]$ and vice versa. In fact, $X$ is f.r. in $Y$ if and
only if $Th\exists(X)\subseteq Th(Y)$, where $Th\exists(X)$ is
the existential theory of $X$.

\begin{Fact}[Corollary~\ref{main1} revisited]  \label{main theorem} Suppose that $X$ is determined by its finite dimensional
subspaces. Then $X$ (isometrically) contains every (separable)
space that is f.r. in $X$.
\end{Fact}

\begin{proof} This is a consequence of Downward
L\"{o}wenheim--Skolem Theorem for continuous logic. Indeed,
suppose that $Y$ is f.r. in $X$. Then, there is an ultrafilter
$\cal U$ such that $Y$ is isometric to a subspace of $X_{\cal
U}$. By Proposition~7.3 in \cite{BBHU}, there is a separable
subspace $Z$ of $X_{\cal U}$ such that $Z$ contains $Y$, and $X$
and $Z$ have the same first order theory. Now, since $X$ is
determined by its finite dimensional subspaces, by
Corollary~\ref{finite->categorical}, $X$ and $Z$ are isometric,
and so $X$ contains $Y$.
\end{proof}

\begin{Remark} {\em  (i) By Dvoretzky's theorem, the Hilbert space $\ell_2$
is finitely representable in every space. So every space which is
(almost) determined by its finite dimensional subspaces (almost)
isometrically contains $\ell_2$.
\newline
(ii) The Hilbert space $\ell_2$ is determined by its finite
dimensional
 subspaces.
\newline(iii) The argument of the proof of Fact~\ref{main theorem} is
similar to the argument due to Ward Henson for existence of
$\ell_2$. We thank professor Henson for communicating to us his
argument. On the other hand, Fact~\ref{main theorem} can be
considered as a consequence of the main theorem of \cite{K}.
\newline
(iv) The above observations (i.e. Theorems \ref{compact->ell_p},
\ref{fr determined->compact}) seem to be new for the Banach
spaces theorists. We thank  William Johnson for useful comments. }
\newline
\end{Remark}

\begin{Question} Is every (separable, infinite-dimensional) Banach space which is
(almost) determined by its finite dimensional subspaces (almost)
isometric to $\ell_2$? (We thank William Johnson and Timothy
Gowers for guiding us to this question.)
\end{Question}


\bigskip\noindent

{\bf Acknowledgements.} I want to thank  William Johnson, Ward
Henson and Timothy Gowers for useful comments.
 I would like to thank the Institute for Basic
Sciences (IPM), Tehran, Iran. Research partially supported by IPM
grants  93030032 and 93030059.

\end{document}